\documentclass[12pt]{article}
\usepackage{amsfonts}

\usepackage{graphicx}
\usepackage{amsmath}

\numberwithin{equation}{section}

\newtheorem{theorem}{Theorem}[section]

\newtheorem{corollary}[theorem]{Corollary}

\newtheorem{lemma}[theorem]{Lemma}

\newenvironment{proof}[1][Proof]{\textbf{#1.} }{\ \rule{0.5em}{0.5em}}

\begin{document}

\title{Backward Stochastic Differential Equations Associated with the
Vorticity Equations}
\author{A. B. Cruzeiro and Z. M. Qian \\
Universidade de Lisboa and Oxford University}
\maketitle

\begin{abstract}
In this paper, we derive a non-linear version of the Feynman-Kac formula for
the solutions of the vorticity equation in dimension 2 with space periodic
boundary conditions. We prove the existence (global in time) and uniqueness
for a stochastic terminal value problem associated with the vorticity
equation in dimension 2.
\end{abstract}

\section{Introduction}

The Feynman-Kac formula, in its original form derived from the idea of path
integration in Feynman's PhD thesis (which is now available in a new print %
\cite{MR2605630}), is a representation formula for solutions of Sch\"{o}%
ingder's equations, and in the hand of Kac, is an explicit formula written
in terms of functional integrals with respect to the Wiener measure, the law
of Brownian motion.

Bismut \cite{MR0453161}, Pardoux-Peng \cite{MR1037747} and Peng \cite%
{MR1149116}, by utilizing It\^{o}'s lemma together with It\^{o}'s martingale
representation, have obtained an interesting non-linear version of
Feynman-Kac's formula for solutions of semi-linear parabolic equations in
terms of backward stochastic differential equations (BSDE). The goal of the
present paper is to derive a Feynman-Kac formula for solutions of the
Navier-Stokes equations in the same spirit of Bismut and Pardoux-Peng \cite%
{MR1037747}, and to study the random terminal problem of the stochastic
differential equations associated with the vorticity equations.

The main idea contained in \cite{MR0453161}, \cite{MR1037747} may be
described as the following. Let $u(t,x)=(u^{1}(t,x)$, $\cdots $, $%
u^{m}(t,x)) $ be a smooth solution to the Cauchy initial value problem of
the following system of semi-linear parabolic equations 
\begin{equation}
\frac{\partial }{\partial t}u^{i}-\nu \Delta u^{i}+f^{i}(u,\nabla u)=0\text{%
, \ }u(0,x)=u_{0}(x)\text{ \ \ in }\mathbb{R}^{d}  \label{in-01}
\end{equation}%
where $i=1,\cdots ,m$, and $\nu >0$ a constant. Let $B=(B^{1},\cdots ,B^{d})$
be the standard Brownian motion on a complete probability space $(\Omega ,%
\mathcal{F},\mathbb{P})$, $x\in \mathbb{R}^{d}$ and $T>0$. Let us read the
solution $u$ along Brownian motion $B$. More explicitly, let $Y_{t}=u(T-t,%
\sqrt{2\nu }B_{t}+x)$ for $t\in \lbrack 0,T]$ and $Z_{t}=\nabla u(T-t,\sqrt{%
2\nu }B_{t}+x)$, and apply It\^{o}'s formula to $u$ and $B$ to obtain 
\begin{equation}
\begin{array}{cc}
Y_{T}-Y_{t}=\int_{t}^{T}f(Y_{s},Z_{s})ds+\sqrt{2\nu }\int_{t}^{T}Z_{s}\cdot
dB_{s}\text{, \ } &  \\ 
Y_{T}=u_{0}(B_{T})\text{. \ \ \ \ \ \ \ \ \ \ \ \ \ \ \ \ \ \ \ \ } & 
\end{array}%
\text{ \ }  \label{in-02}
\end{equation}%
In literature, (\ref{in-01}) may be written in differential form 
\begin{equation}
dY=f(Y,Z)dt+\sqrt{2\nu }Z\cdot dB\text{, }Y_{T}=\xi \text{, \ }
\label{in-03}
\end{equation}%
where the arguments $s$, $t$ etc. are suppressed if no confusion may arise.
The differential equation above is an example of backward stochastic
differential equations, where the terminal value $Y_{T}=\xi $ is given. The
function $f$ appearing on the right hand side of (\ref{in-03}) is called the
(non-linear) \emph{driver}.

Pardoux-Peng \cite{MR1037747} made an important observation. If the
non-linear driver $f$ in BSDE\ (\ref{in-03}) is globally Lipschitz
continuous, then there is a unique adapted solution pair $(Y,Z)$ satisfying (%
\ref{in-03}) for a random terminal value $\xi \in L^{2}(\Omega ,\mathcal{F}%
_{T},\mathbb{P})$, which is not necessary in the form of $u_{0}(B_{T})$. The
solution $u$ and its gradient $\nabla u$ in turn can be represented in terms
of $(Y,Z)$. This representation may be considered as a nonlinear extension
of Feynman-Kac's formula to semi-linear parabolic equations.

More recently, Kobylanski \cite{MR1782267}, Delarue \cite{MR2053051},
Briand-Hu \cite{MR2391164}, Tevzadze \cite{MR2389055}, and etc. have
extended Pardoux-Peng's result to some BSDEs with non-linear drivers of
quadratic growth. These papers however mainly deal with scalar BSDEs only,
which corresponds to semi-linear scalar parabolic equations. It remains
largely an open problem whether the BSDE approach may be applied to
non-parabolic type of partial differential equations. We study in the
present paper a class of backward stochastic differential equations which
arise from the vorticity formulation of the Navier-Stokes equations, hence
provide Feynman-Kac type formula for solutions of the Navier-Stokes
equations.

Relations between the Navier-Stokes equation and foward-backward stochastic
differential equations formulated in the group of diffeomorphisms were
introduced in \cite{CS1}. 

\section{The vorticity equation}

Let us describe a class of (infinite dimensional) backward stochastic
differential equations associated with the study of Navier-Stokes equations.

The 2D Navier-Stokes equations (without external force) are the partial
differential equations which describe the motion of fluids 
\begin{equation}
\begin{array}{cc}
\frac{\partial u}{\partial t}-\nu \Delta u+u\cdot \nabla u+\nabla p=0\text{,}
&  \\ 
\nabla \cdot u=0\text{,} & 
\end{array}
\label{e12-01}
\end{equation}%
where $u=(u^{1},u^{2})$ is the velocity field, $\nu $ the viscosity constant
and $p$ the pressure. The mathematical study of the Navier-Stokes is
interesting by its own, and even the simplest situation where the space
periodic condition is supplied is of interest.

Suppose that $u(0,x)=\varphi (x)$ is a smooth vector field with period one,
that is, $\varphi (x+e_{i})=\varphi (x)$ for all $x\in \mathbb{R}^{2}$,
where $e_{1}=(1,0)$ and $e_{2}=(0,1)$ the standard basis in $\mathbb{R}^{2}$%
. Then, the unique solution $(u,p)$ to the 2D Navier-Stokes equation is
smooth on $(0,\infty )\times \mathbb{R}^{2}$ and periodic in space
variables, so that $u(t,x+e_{i})=u(t,x)$ and $p(t,x+e_{i})=p(t,x)$ for all $%
t>0$, $x\in \mathbb{R}^{2}$, $i=1,2$.

Let 
\begin{equation}
\omega =\frac{\partial u^{2}}{\partial x^{1}}-\frac{\partial u^{1}}{\partial
x^{2}}  \label{e12-03}
\end{equation}%
be the vorticity of $u$, which is a scalar function in dimension two, and
thus the evolution equation for $\omega $ is a scalar partial differential
equation 
\begin{equation}
\frac{\partial \omega }{\partial t}-\nu \Delta \omega +u\cdot \nabla \omega
=0\text{.}  \label{e12-02}
\end{equation}%
Equation (\ref{e12-02}) is called the vorticity equation is equivalent to
the Navier-Stokes equation. The relationship between the scalar function $%
\omega $ and the associated vector field $u$ is determined by the Poisson
equations 
\begin{equation}
\Delta u^{1}=-\frac{\partial \omega }{\partial x^{2}}\text{ \ and \ }\Delta
u^{2}=\frac{\partial \omega }{\partial x^{1}}\text{.}  \label{e12-04}
\end{equation}%
By (\ref{e12-03}) the average of $\omega (t,x)$%
\begin{equation}
\int_{\lbrack 0,1)^{2}}\omega (t,x)dx=0\text{ \ for all }t\geq 0\text{,}
\label{e12-05}
\end{equation}%
so that (\ref{e12-04}) has a unique periodic (with period one) solution $%
u=(u^{1},u^{2})$. We define linear operators $K_{i}:\omega \rightarrow u^{i}$
(where $i=1,2$) and $K:\omega \rightarrow u$ by solving the Poisson
equations (\ref{e12-04}), where $\omega $ is a real function with period one
and mean zero.

Let $\mathbb{T}^{2}=\mathbb{R}^{2}/\mathbb{Z}^{2}$ be the 2D torus equipped
with the standard metric and the Lebesgue measure. We may identify tensor
fields in $\mathbb{R}^{2}$ with period one canonically with the
corresponding tensor fields on $\mathbb{T}^{2}$. For example 
\begin{equation*}
L^{2}(\mathbb{T}^{2})=\{f\in L_{\text{loc}}^{2}(\mathbb{R}^{2}):f(\cdot
+e_{i})=f(\cdot )\text{ for }i=1,2\}\cap L^{2}([0,1)^{2})\text{.}
\end{equation*}
If $f\in L^{2}(\mathbb{T}^{2})$ then 
\begin{equation}
f(x)=\sum_{k\in \mathbb{Z}^{2}}e^{2\pi \sqrt{-1}\langle k,x\rangle }\hat{f}%
(k)  \label{e12-09}
\end{equation}
where 
\begin{equation}
\hat{f}(k)=\int_{[0,1)^{2}}e^{-2\pi \sqrt{-1}\langle k,y\rangle }f(y)dy\text{%
, \ \ \ }k\in \mathbb{Z}^{2}  \label{e12-10}
\end{equation}
is the Fourier transform of $f$, $\langle \cdot ,\cdot \rangle $ denotes the
scalar product in Euclidean spaces.

\begin{lemma}
(Green's formula) Consider the Poisson equation 
\begin{equation}
\Delta g=-f\text{ \ \ \ in }\mathbb{T}^{2}\text{, \ }\int_{\mathbb{T}%
^{2}}g(y)dy=0\text{,}  \label{e12-11}
\end{equation}
where $\int_{\mathbb{T}^{2}}f(y)dy=0$ and $f\in L^{2}(\mathbb{T}^{2})$. Then
the unique solution of the problem (\ref{e12-11}) is given by 
\begin{equation}
g(x)=\sum_{k\in \mathbb{Z}^{2},k\neq 0}\frac{e^{2\pi \sqrt{-1}\langle
k,x\rangle }}{4\pi |k|^{2}}\hat{f}(k)\text{ .}  \label{e12-12}
\end{equation}
\end{lemma}

Our first goal is to derive a probabilistic representation for $\omega $ in
terms of Brownian motion. To this end we set up the probability setting with
which we are going to work with. Let $B=(B^{1},B^{2})$ be a standard
Brownian motion on a complete probability space $(\Omega ,\mathcal{F},P)$,
and define 
\begin{equation*}
\begin{array}{ccc}
Y(w,t,x) & = & \omega (T-t,x+\sqrt{2v}B_{t}(w))\text{,} \\ 
Z^{1}(w,t,x) & = & \frac{\partial \omega }{\partial x^{1}}(T-t,x+\sqrt{2v}%
B_{t}(w))\text{,} \\ 
Z^{2}(w,t,x) & = & \frac{\partial \omega }{\partial x^{2}}(T-t,x+\sqrt{2v}%
B_{t}(w))%
\end{array}%
\text{ \ }
\end{equation*}%
for $(w,t,x)\in \Omega \times \lbrack 0,\infty )\times \mathbb{R}^{2}$. We
will often suppress the random element $w$ from our notations, and write $%
Y(t,x)$, $Y_{t}$ or simply by $Y$ for $Y(w,t,x)$, if no confusion is
possible.

Let $\psi =\frac{\partial \varphi _{2}}{\partial x^{1}}-\frac{\partial
\varphi _{1}}{\partial x^{2}}$ be the vorticity of the initial velocity $%
\varphi \equiv u_{0}$, and $\xi (x)=\psi (x+\sqrt{2v}B_{T})$. Then, it is
clear that $\xi $ is smooth and periodic in $x$ (with again period one).
According to It\^{o}'s formula 
\begin{eqnarray}
\xi (x)-Y(t,x) &=&\sqrt{2v}\int_{t}^{T}\langle \nabla \omega (T-s,x+\sqrt{2v}%
B_{s}),dB_{s}\rangle  \notag \\
&&+\int_{t}^{T}\left( -\frac{\partial \omega }{\partial s}+\nu \Delta \omega
\right) (T-s,x+\sqrt{2v}B_{s})ds\text{.}  \label{e12-07}
\end{eqnarray}%
Now, by utilizing the vorticity equation (\ref{e12-02}): substitute $-\frac{%
\partial \omega }{\partial s}+\nu \Delta \omega $ by $u\cdot \nabla \omega $
to obtain 
\begin{equation}
\xi (x)-Y(t,x)=\int_{t}^{T}\langle Z(s,x),X(s,x)\rangle ds+\sqrt{2v}%
\int_{t}^{T}\langle Z(s,x),dB_{s}\rangle  \label{e12-08}
\end{equation}%
where for simplicity we have set 
\begin{equation*}
X(t,x)=u(T-t,x+\sqrt{2v}B_{t})
\end{equation*}%
which is continuous in $t$, smooth in $x$, and periodic in $x$. Next, we
wish to rewrite $X(t,x)$ in terms of $Y$ and $Z$. To this end, it is a good
idea to introduce some notions in Fourier analysis, and establish several
notations which will be used in what follows.

Let us apply Green's formula to the vorticity $\omega $ of $u$. According to
(\ref{e12-04}) and (\ref{e12-12}), we have 
\begin{equation}
u^{1}(x)=\frac{\sqrt{-1}}{2}\sum_{k=(k_{1},k_{2})\in \mathbb{Z}^{2},k\neq 0}%
\frac{k_{2}}{|k|^{2}}e^{2\pi \sqrt{-1}\langle k,x\rangle }\widehat{\omega }%
(k)  \label{e12-13}
\end{equation}
and 
\begin{equation}
u^{2}(x)=-\frac{\sqrt{-1}}{2}\sum_{k=(k_{1},k_{2})\in \mathbb{Z}^{2},k\neq 0}%
\frac{k_{1}}{|k|^{2}}e^{2\pi \sqrt{-1}\langle k,x\rangle }\widehat{\omega }%
(k)\text{.}  \label{e12-14}
\end{equation}
In other words 
\begin{equation}
\widehat{u^{1}}(k)=\frac{\sqrt{-1}}{2}\frac{k_{2}}{|k|^{2}}\widehat{\omega }%
(k)\text{ \ and }\widehat{u^{2}}(k)=-\frac{\sqrt{-1}}{2}\frac{k_{1}}{|k|^{2}}%
\widehat{\omega }(k)\text{, \ \ }k\neq 0\text{.}  \label{e12-15}
\end{equation}
Hence 
\begin{eqnarray}
X^{1}(t,x) &=&u^{1}(T-t,x+\sqrt{2v}B_{t})  \notag \\
&=&\frac{\sqrt{-1}}{2}\sum_{k=(k_{1},k_{2})\in \mathbb{Z}^{2},k\neq 0}\frac{%
k_{2}}{|k|^{2}}e^{2\pi \sqrt{-1}\langle k,x+\sqrt{2v}B_{t}\rangle }\widehat{%
\omega (T-t,\cdot )}(k)\text{.}  \label{e12-17}
\end{eqnarray}

On the other hand 
\begin{eqnarray*}
\widehat{Y(t,\cdot )}(k) &=&\int_{[0,1)^{2}}e^{-2\pi \sqrt{-1}\langle
k,y\rangle }\omega (T-t,y+\sqrt{2v}B_{t})dy \\
&=&e^{2\pi \sqrt{-1}\langle k,\sqrt{2v}B_{t}\rangle }\int_{[0,1)^{2}+\sqrt{2v%
}B_{t}}e^{-2\pi \sqrt{-1}\langle k,y\rangle }\omega (T-t,y)dy \\
&=&e^{2\pi \sqrt{-1}\langle k,\sqrt{2v}B_{t}\rangle
}\int_{[0,1)^{2}}e^{-2\pi \sqrt{-1}\langle k,y\rangle }\omega (T-t,y)dy \\
&=&e^{2\pi \sqrt{-1}\langle k,\sqrt{2v}B_{t}\rangle }\widehat{\omega
(T-t,\cdot )}(k)
\end{eqnarray*}%
here the third equality follows from the fact that $\omega $ is periodic.
Substituting the above equality into (\ref{e12-17}) to obtain 
\begin{equation}
X^{1}(t,x)=\frac{\sqrt{-1}}{2}\sum_{k=(k_{1},k_{2})\in \mathbb{Z}^{2},k\neq
0}\frac{k_{2}}{|k|^{2}}e^{2\pi \sqrt{-1}\langle k,x\rangle }\widehat{%
Y(t,\cdot )}(k)\text{,}  \label{e12-18}
\end{equation}%
and 
\begin{equation}
X^{2}(t,x)=-\frac{\sqrt{-1}}{2}\sum_{k\neq (k_{1},k_{2})\in \mathbb{Z}%
^{2},k\neq 0}\frac{k_{1}}{|k|^{2}}e^{2\pi \sqrt{-1}\langle k,x\rangle }%
\widehat{Y(t,\cdot )}(k)\text{.}  \label{e12-19}
\end{equation}%
By our definition of linear operators $K_{i}$, the previous equations (\ref%
{e12-18}, \ref{e12-19}) may be written as 
\begin{equation}
X^{j}(t,x)=K_{j}(Y(t,\cdot ))(x)\text{ \ \ \ \ \ \ }\forall x\in \mathbb{R}%
^{2}\text{, }j=1,2\text{.}  \label{e12-20}
\end{equation}%
Thanks to (\ref{e12-20}) we may finish our computation for $Y$ as following.
According to (\ref{e12-08}) 
\begin{equation}
\xi (x)-Y(t,x)=\sqrt{2v}\int_{t}^{T}\langle Z(s,x),dB_{s}\rangle
+\int_{t}^{T}\langle Z(s,x),K(Y(s,\cdot ))(x)\rangle ds  \label{e12-21}
\end{equation}%
where $x$ runs through $\mathbb{R}^{2}$.

\section{Feynman-Kac formula for the vorticity}

The preceding stochastic integral equation (\ref{e12-21}) may be put in a
differential form 
\begin{equation}
dY=\langle Z,K(Y)\rangle dt+\sqrt{2v}\langle Z,dB\rangle \text{, \ \ }%
Y_{T}=\xi  \label{b-12-01}
\end{equation}%
where the time space variable $x$, for simplicity, is suppressed. The
initial value problem to the vorticity equation (\ref{e12-02}) is
transferred to a terminal value problem to the stochastic differential
equation (\ref{b-12-01}) within the formulation of BSDEs. In order to derive
a non-linear version of the Feynman-Kac formula for the vorticity $\omega $,
we need to study the infinite dimensional BSDE (\ref{b-12-01}).

BSDE (\ref{b-12-01}) possesses two features which make it difficulty to
study. First, the stochastic equation (\ref{b-12-01}) must be solved in a
function space, so it is an infinite dimensional stochastic differential
equation (with finite dimensional noise). Second, the non-linear term in
this BSDE is quadratic in $Y$ and $Z$, which is the origin of all
difficulties. There are few results in literature about this kind of
backward stochastic differential equations.

Let $B=(B^{1},B^{2})$ be a standard Brownian motion on a complete
probability space $(\Omega ,\mathcal{F},\mathbb{P})$. Let $\mathcal{F}%
_{t}^{0}=\sigma \{B_{s}$: $s\leq t\}$ and $(\mathcal{F}_{t})_{t\geq 0}$ be
the completed continuous filtration associated with $(\mathcal{F}%
_{t}^{0})_{t\geq 0}$. Let $\mathcal{O}$ and $\mathcal{P}$ be the optional
and predictable $\sigma $-fields on $\Omega \times \lbrack 0,\infty )$,
respectively. Let $\mathcal{\tilde{Q}}=\mathcal{O}\times \mathcal{B}(\mathbb{%
R}^{2})$ and $\mathcal{\tilde{P}}=\mathcal{P}\times \mathcal{B}(\mathbb{R}%
^{2})$ be the optional and predictable $\sigma $-algebras on $\Omega \times
\lbrack 0,\infty )\times \mathbb{R}^{2}$. A $\mathcal{\tilde{Q}}$-measurable
(resp. $\mathcal{\tilde{P}}$-measurable) function on $\Omega \times \lbrack
0,\infty )\times \mathbb{R}^{2}$ is called a optional (resp. predictable)
function, or called an $\mathbb{R}^{2}$-valued optional (resp. predictable).
\ We may similarly define $\mathcal{O}\times \mathcal{B}(\mathbb{T}^{2})$
and $\mathcal{P}\times \mathcal{B}(\mathbb{T}^{2})$ which are identified
with elements in the $\mathcal{O}\times \mathcal{B}(\mathbb{R}^{2})$ and $%
\mathcal{P}\times \mathcal{B}(\mathbb{R}^{2})$ respectively which are
periodic in the space variables with period one.

In order to derive a non-linear version of the Feynman-Kac formula for the
vorticity $\omega $ we need to prove the existence and the uniqueness of
solutions to (\ref{b-12-01}) subject to the given terminal value $\xi $.
Actually we will do this for a general terminal value $\xi $ which is not
necessary in a form of $\varphi (B_{T})$.

We will assume that $\xi $ is a \emph{bounded random function} on $\Omega
\times \mathbb{T}^{2}$ which is $\mathcal{F}_{T}\otimes \mathcal{B}(\mathbb{T%
}^{2})$ measurable, and furthermore, we assume that for every $w\in \Omega $%
, $\xi (w,\cdot )\in W^{2,2}(\mathbb{T}^{2})$, and $\int_{\mathbb{T}^{2}}\xi
(w,y)dy=0$ for all $w\in \Omega $. In particular, according to the Sobolev
imbedding, $x\rightarrow \xi (w,x)$ is continuous on $\mathbb{T}^{2}$ for
every $w\in \Omega $.

By a solution to BSDE (\ref{b-12-01}) we mean a pair of $\mathcal{\tilde{P}}$%
-measurable stochastic processes $Y$ and $Z$, such that:

1) For all $x\in \mathbb{T}^{2}$, $(w,t)\rightarrow Y(w,t,x)$ is continuous
semimartingale, and for all $(w,t)\in \Omega \times \lbrack 0,T]$, $%
Y(w,t,\cdot )\in L^{2}(\mathbb{T}^{2})$.

2) For every $x\in \mathbb{T}^{2}$, 
\begin{equation*}
\mathbb{E}\int_{0}^{T}|Z(t,x)|^{2}dt<+\infty
\end{equation*}
so that the It\^{o}'s integral $\int_{0}^{\cdot }\langle Z(\cdot
,x),dB\rangle $ is a square integrable martingale for every $x$.

3) It holds that 
\begin{equation*}
Y(t,x)=\xi (x)-\int_{t}^{T}\langle Z(s,x),K(Y(s,\cdot ))(x)\rangle ds-\sqrt{%
2v}\int_{t}^{T}\langle Z(s,x),dB_{s}\rangle
\end{equation*}
almost surely on $\Omega \times \mathbb{T}^{2}$, for $t\in \lbrack 0,T]$.

Now we are in a position to state our main result.

\begin{theorem}
\label{main-th}Under above assumptions on the terminal value $\xi $, there
is a unique solution pair $(Y,Z)$ to BSDE\ (\ref{b-12-01}) such that

1) $Y$ is bounded on $\Omega \times \lbrack 0,T]\times \mathbb{T}^{2}$, and

2) For almost all $x\in \mathbb{T}^{2}$, the It\^{o} integral $%
\int_{0}^{\cdot }\langle Z(\cdot ,x),dB\rangle $ is a BMO martingale, and 
\begin{equation*}
\text{ess}\sup_{[0,T]\times \Omega }\mathbb{E}\left\{ \left.
\int_{t}^{T}||Z_{s}||^{2}ds\right| \mathcal{F}_{t}\right\} <+\infty
\end{equation*}%
where $||\cdot ||$ denotes the $L^{2}$-norm on $\mathbb{T}^{2}$.
\end{theorem}

In particular, by applying Theorem \ref{main-th} to $\xi =\varphi (B_{T})$
where $\varphi =\nabla \times u_{0}$ is bounded, $C^{2}$ on $\mathbb{T}^{2}$%
, then $Y(t,x)=\omega (T-t,\sqrt{2\nu }B_{t}+x)$, where $(Y,Z)$ is the
unique solution pair of (\ref{b-12-01}) with terminal $\xi =\varphi (B_{T})$
and $\omega $ is the solution to the vorticity equation (\ref{e12-02})
subject to the initial value $u(0,\cdot )=u_{0}$. $Y$ may be regarded as the
probabilistic representation for the vorticity $\omega $.

The proof of Theorem \ref{main-th} relies on two important technical facts.
The first is the $L^{2}$-estimate for the linear operator $K$, and the
secomd is a maximal principle formulated in terms of backward stochastic
differential equations.

\section{Several technical estimates}

In order to prove the main result Theorem \ref{main-th}, we need several a
priori estimates.

\subsection{\textit{Apriori} estimates for $K$}

Let us recall the definition of $K$. Note that we identify tensor fields in
the torus $\mathbb{T}^{2}$ with the tensor fields on $\mathbb{R}^{2}$ with
period one along each space variable. For $k\in \mathbb{Z}_{+}$ and $q\geq 1$
the Sobolev space 
\begin{eqnarray*}
W^{k,q}(\mathbb{T}^{2}) &=&\left\{ f:\partial ^{\alpha }f\in L_{\text{loc}%
}^{q}(\mathbb{R}^{2})\cap L^{q}([0,1)^{2})\text{ for }|\alpha |\leq k\right.
\\
&&\text{ }\left. \text{and }f(\cdot +e_{i})=f(\cdot )\text{ \ for }%
i=1,2\right\}
\end{eqnarray*}
together with the Sobolev norm 
\begin{equation*}
||f||_{k,q}=\left( \sum_{\alpha \in \mathbb{Z}^{2},|\alpha |\leq
k}||\partial ^{\alpha }f||_{q}^{q}\right) ^{1/q}
\end{equation*}
where $||\cdot ||_{q}$ is the $L^{q}$-norm over $\mathbb{T}^{2}$, that is 
\begin{equation*}
||f||_{q}=\left( \int_{\mathbb{T}^{2}}|f|^{q}\right) ^{1/q}=\left(
\int_{[0,1)^{2}}|f(x)|^{q}dx\right) ^{1/q}\text{.}
\end{equation*}
If $q=2$ then we use $||\cdot ||$ instead of $||\cdot ||_{2}$ for simplicity.

According to Sobolev's embedding theorem, $W^{2,2}(\mathbb{T}%
^{2})\hookrightarrow C^{\alpha }(\mathbb{T}^{2})$ for some $\alpha \in (0,1) 
$, so any element in $W^{2,2}(\mathbb{T}^{2})$ has a unique continuous
representation.

If $f\in L^{2}(\mathbb{T}^{2})$ such that $\int_{[0,1)^{2}}f=0$, then $%
K_{j}(f)=g_{j}$ are the unique solutions (with period one) such that $%
\int_{[0,1)^{2}}g_{j}=0$, solving the Poisson equations 
\begin{equation}
\Delta g_{1}=-\frac{\partial f}{\partial x_{2}}\text{, \ }\Delta g_{2}=\frac{%
\partial f}{\partial x_{1}}\text{ \ \ on }\mathbb{T}^{2}\text{.}
\label{e12-23}
\end{equation}
By definition, if $\alpha =(\alpha _{1},\alpha _{2})\in \mathbb{Z}_{+}\times 
\mathbb{Z}_{+}$, then $\partial ^{\alpha }K_{j}(f)=K_{j}(\partial ^{\alpha
}f)$, where $\partial ^{\alpha }$ stands for the partial derivative $\frac{%
\partial ^{|\alpha |}}{\partial x_{1}^{\alpha _{1}}\partial x_{2}^{\alpha
_{2}}}$ for simplicity as long as $\partial ^{\alpha }f\in L^{2}(\mathbb{T}%
^{2})$.

On the other hand 
\begin{equation*}
\int_{\mathbb{T}^{2}}|\nabla g_{j}|^{2}=-\int_{\mathbb{T}^{2}}g_{j}\Delta
g_{j}=\int_{\mathbb{T}^{2}}g_{1}\frac{\partial f}{\partial x_{2}}\text{ or }%
-\int_{\mathbb{T}^{2}}g_{2}\frac{\partial f}{\partial x_{1}}
\end{equation*}%
according to $j=1$ or $j=2$. Integrating by parts together with
Cauchy-Schwartz inequality to the last integrals we deduce that 
\begin{equation*}
\int_{\mathbb{T}^{2}}|\nabla g_{j}|^{2}\leq \sqrt{\int_{\mathbb{T}%
^{2}}|\nabla g_{j}|^{2}}\sqrt{\int_{\mathbb{T}^{2}}|f|^{2}}
\end{equation*}%
which yields that 
\begin{equation}
||\nabla K_{j}(f)||\leq ||f||\text{, \ \ \ \ }j=1,2\text{.}  \label{e12-24}
\end{equation}%
Let $\lambda _{1}>0$ be the spectral gap for the torus $\mathbb{T}^{2}$.
Since $\int_{\mathbb{T}^{2}}K_{j}(f)=0$, according to the Poincar\'{e}
inequality 
\begin{equation}
||K_{j}(f)||\leq \frac{1}{\sqrt{\lambda _{1}}}||\nabla K_{j}(f)||\leq \frac{1%
}{\sqrt{\lambda _{1}}}||f||\text{.}  \label{e12-25}
\end{equation}%
Therefore we have the following elliptic estimate (for more details see for
example \cite{MR0125307}, \cite{MR0162050} and \cite{MR0202511}).

\begin{lemma}
There is a universal constant $C_{0}>0$ such that 
\begin{equation*}
||K_{j}(f)||_{k,2}\leq C_{0}||f||_{k-1,2}
\end{equation*}
for any $f\in W^{k-1,2}(\mathbb{T}^{2})$ with $\int_{\mathbb{T}^{2}}f=0$,
where $k\in \mathbb{N}$.
\end{lemma}

In particular, if $f\in W^{1,2}(\mathbb{T}^{2})$, $K(f)$ is $\alpha $-H\"{o}%
lder continuous.

\subsection{A maximum principle}

Let us formulate a probabilistic version of the maximum principle in terms
of BSDE.

\begin{lemma}
\label{lem-max-1}Suppose $y$ is a continuous semimartingale such that 
\begin{equation*}
y_{t}=y_{T}-\int_{t}^{T}\langle h,z\rangle ds-\int_{t}^{T}\langle
z,dB\rangle \text{ \ \ for }t\in \lbrack 0,T]\text{,}
\end{equation*}
where $y_{T}$ is a bounded $\mathcal{F}_{T}$-measurable random variable,
both $z$ and $h$ are $\mathbb{R}^{d}$-valued predictable processes such that 
\begin{equation*}
\mathbb{E}\int_{0}^{T}|z|^{2}dt<\infty
\end{equation*}
and suppose that 
\begin{equation*}
R_{t}=\exp \left[ -\int_{0}^{t}\langle h,dB\rangle -\frac{1}{2}%
\int_{0}^{t}|h|^{2}ds\right]
\end{equation*}
is a martingale up to $T$. Then $|y_{t}|_{\infty }\leq |y_{T}|_{\infty }$
for all $t\in \lbrack 0,T]$ almost surely.
\end{lemma}

\begin{proof}
Define a probability $\mathbb{Q}$ with density $R$. Then $\mathbb{P}$ is
equivalent to $\mathbb{Q}$ on $\mathcal{F}_{T}$, and according to Girsanov's
theorem $\tilde{B}_{t}=B_{t}+\int_{0}^{t}h_{s}ds$ is a standard Brownian
motion, and 
\begin{equation*}
y_{t}-y_{T}=-\int_{t}^{T}\langle z,d\tilde{B}\rangle \text{.}
\end{equation*}
Therefore $y_{t}=\mathbb{E}^{\mathbb{Q}}\left\{ y_{T}|\mathcal{F}%
_{t}\right\} $ so that $|y_{t}|_{\infty }\leq |y_{T}|_{\infty }$ almost
surely.
\end{proof}

\subsection{A linear BSDE}

Let us consider the following linear BSDE 
\begin{equation}
dY=\langle Z,h\rangle dt+\langle Z,dB\rangle \text{, \ }Y_{T}=\xi
\label{l-12-24-1}
\end{equation}%
with $h\in \mathcal{\tilde{Q}}$ is a given optional process (valued in $%
\mathbb{T}^{2}$) such that for each $(w,t)\in \Omega \times \lbrack 0,T]$, $%
h(w,t,\cdot )\in C(\mathbb{T}^{2})$ and 
\begin{equation}
\mathbb{E}\int_{0}^{T}\left| h(t,x)\right| ^{2}dt<\infty \text{ \ \ }\forall
x\in \mathbb{T}^{2}\text{.}  \label{l-12-24-5}
\end{equation}%
$\xi $ is the terminal value: 
\begin{equation*}
\xi \in L^{\infty }(\Omega \times \mathbb{T}^{2})\cap L^{\infty }(\Omega ,%
\mathcal{F}_{T},W^{2,2}(\mathbb{T}^{2}))\text{.}
\end{equation*}%
The linear equation (\ref{l-12-24-1}) may be solved for every $x\in \mathbb{T%
}^{2}$, and in fact we may solve the linear BSDE 
\begin{equation}
\begin{array}{ccc}
dY(t,x) & = & \langle Z(t,x),h(t,x)\rangle dt+\langle Z(t,x),dB_{t}\rangle 
\text{,} \\ 
Y(T,x) & = & \xi (x)\text{, \ \ \ \ \ \ \ \ \ \ \ \ \ \ \ \ \ \ \ \ \ \ \ }%
\end{array}
\label{l-12-24-3}
\end{equation}%
by means of changing probability. More precisely, for each $x\in \mathbb{T}%
^{2}$, since (\ref{l-12-24-5}) holds, we can define a probability $\mathbb{Q}%
^{x}$ on $\mathcal{F}_{T}$ by $\frac{d\mathbb{Q}^{x}}{d\mathbb{P}}=R(T,x)$,
where 
\begin{equation*}
R(t,x)=\exp \left[ -\int_{0}^{t}\langle h(s,x),dB_{s}\rangle -\frac{1}{2}%
\int_{0}^{t}|h(s,x)|^{2}ds\right] \text{.}
\end{equation*}%
If $(Y(\cdot ,x),Z(\cdot ,x))$ is the unique solution of (\ref{l-12-24-3}),
then, according to the Girsanov theorem, $Y(\cdot ,x)$ must be a martingale
under the new probability $\mathbb{Q}^{x}$, we therefore have 
\begin{equation*}
Y(t,x)=\mathbb{E}^{\mathbb{Q}^{x}}\left\{ \xi (x)|\mathcal{F}_{t}\right\}
\end{equation*}%
which implies that 
\begin{equation*}
Y(t,x)=\mathbb{E}\left\{ \left. \frac{R(T,x)}{R(t,x)}\xi (x)\right| \mathcal{%
F}_{t}\right\} \text{ \ \ }
\end{equation*}%
for $(t,x)\in \lbrack 0,T]\times \mathbb{T}^{2}$. Therefore we have
established the following

\begin{lemma}
\label{lem-q1} Suppose that $\xi $ is $W^{2,2}(\mathbb{T}^{2})$-valued $%
\mathcal{F}_{T}$-measurable random variable, and suppose that $h$ is a $%
W^{2,2}(\mathbb{T}^{2})$-valued adapted stochastic process satisfying (\ref%
{l-12-24-5}), then the unique solution to (\ref{l-12-24-1}) is given by 
\begin{equation}
Y(t,x)=\mathbb{E}\left\{ \left. \xi (x)e^{-\int_{t}^{T}\langle
h(s,x),dB_{s}\rangle -\frac{1}{2}\int_{t}^{T}|h(s,x)|^{2}ds}\right| \mathcal{%
F}_{t}\right\} \text{ \ \ }  \label{l-12-24-6}
\end{equation}
for $t\in \lbrack 0,T]\times \mathbb{T}^{2}$.
\end{lemma}

It is clear that 
\begin{eqnarray*}
\partial _{j}Y(t,x) &=&\mathbb{E}\left\{ \left. \partial _{j}\xi (x)\frac{%
R(T,x)}{R(t,x)}\right| \mathcal{F}_{t}\right\} \\
&&+\mathbb{E}\left\{ \left. \xi (x)\left( -\int_{t}^{T}\langle \partial
_{j}h(s,x),dB_{s}\rangle -\int_{t}^{T}\langle h(s,x),\partial
_{j}h(s,x)\rangle ds\right) \frac{R(T,x)}{R(t,x)}\right| \mathcal{F}%
_{t}\right\}
\end{eqnarray*}%
so we have the following simple fact.

\begin{corollary}
\label{lem-dec1}1) If in addition $\xi $ and $h$ are bounded, then the
solution $Y$ is continuous in $(t,x)$. \ 2) If in addition $\xi $ and $h$
have bounded derivatives in $x$, then so is $Y$.
\end{corollary}

\section{Proof of Theorem \ref{main-th}}

This section is devoted to the proof of Theorem \ref{main-th}. We will use
the following convention in our proof. The elliptic estimates show that if $%
f\in W^{k,2}(\mathbb{T}^{2})$ then $K(f)\in W^{k+1,2}(\mathbb{T}^{2})$, thus
if $k\geq 1$, then, according to the Sobolev imbedding, $K(f)$ has a H\"{o}%
lder continuous version. Therefore, if $f\in W^{1,2}(\mathbb{T}^{2})$ for $%
k\geq 1$, $K(f)$ is always chosen to be its continuous version.

Let $\mathcal{H}$ denote the collection of all bounded $\mathcal{\tilde{P}}$%
-measurable stochastic processes $Y$ on $\Omega \times \lbrack 0,T]\times 
\mathbb{T}^{2}$ satisfying the following conditions:

1) For each $x\in \mathbb{T}^{2}$, $Y(\cdot ,x)$ is a continuous
semimartingale (up to time $T$) on $(\Omega ,\mathcal{F},\mathcal{F}_{t},%
\mathbb{P})$ whose martingale part $M$ (with $M_{0}=0$) is a BMO martingale,
and $Y_{T}=\xi $. Moreover, for every $(w,t)\in \Omega \times \lbrack 0,T]$, 
$Y(w,t,\cdot )\in W^{2,2}(\mathbb{T}^{2})$.

2) Let the It\^{o} representation of the martingale part 
\begin{equation*}
M(t,x)=\int_{0}^{t}\langle Z(t,x),dB_{t}\rangle
\end{equation*}
where $Z$ is $\mathcal{\tilde{P}}$-measurable. Then 
\begin{equation*}
\text{ess}\sup_{[0,T]\times \Omega }\mathbb{E}\left\{ \left.
\int_{t}^{T}||Z_{s}||^{2}ds\right| \mathcal{F}_{t}\right\} <+\infty \text{.}
\end{equation*}

Let $Y\in \mathcal{H}$, we define $\mathcal{L}(Y)=\tilde{Y}$ by solving the
following \emph{linear} backward stochastic differential equation 
\begin{equation}
d\tilde{Y}(t,x)=\langle \tilde{Z}(t,x),K(Y(t,\cdot ))(x)\rangle dt+\sqrt{2v}%
\langle \tilde{Z}(t,x),dB_{t}\rangle \text{, }\tilde{Y}(T,x)=\xi (x)\text{.}
\label{7-9-e2}
\end{equation}
for every $x\in \mathbb{T}^{2}$. Then $\tilde{Y}\in \mathcal{H}$.

\subsection{\textit{Apriori} estimate for the density process $Z$}

If $Y\in \mathcal{H}$, in particular $Y$ is a bounded function on $\Omega
\times \lbrack 0,T]\times \mathbb{T}^{2}$. $||Y||_{\infty }$ denotes the
essential bound of $Y$ on $\Omega \times \lbrack 0,T]\times \mathbb{T}^{2}$.

Suppose $Y\in \mathcal{H}$ such that $||Y||_{\infty }\leq C_{1}$. Define $%
\tilde{Y}=\mathcal{L}(Y)$, and $\tilde{Z}$ is the density process of the
martingale part of $\tilde{Y}$, that is, $(\tilde{Y},\tilde{Z})$ by solving
the following linear BSDE 
\begin{equation}
d\tilde{Y}=\langle \tilde{Z},K(Y)\rangle ds+\sqrt{2\nu }\langle \tilde{Z}%
,dB\rangle \text{, }\tilde{Y}_{T}=\xi  \label{6-2-1}
\end{equation}
where $|\xi (w,t,x)|\leq C_{1}$. By the maximal principle, $|\tilde{Y}%
(w,t,x)|\leq C_{1}$.

For simplicity, we will use $\mathbb{E}^{\mathcal{F}_{t}}$ to denote the
conditional expectation $\mathbb{E}\{\cdot |\mathcal{F}_{t}\}$.

By It\^{o}'s calculus,

\begin{eqnarray*}
|\tilde{Y}_{t}|^{2} &=&|\xi |^{2}-2\nu \int_{t}^{T}|\tilde{Z}%
|^{2}ds-2\int_{t}^{T}\tilde{Y}\langle \tilde{Z},K(Y)\rangle ds \\
&&-2\sqrt{2\nu }\int_{t}^{T}\tilde{Y}\langle \tilde{Z},dB\rangle \text{ .}
\end{eqnarray*}
First take conditional expectations, to obtain that

\begin{eqnarray*}
|\tilde{Y}_{t}|^{2}+2\nu \mathbb{E}^{\mathcal{F}_{t}}\int_{t}^{T}|\tilde{Z}%
|^{2}ds &=&\mathbb{E}^{\mathcal{F}_{t}}|\xi |^{2}-2\mathbb{E}^{\mathcal{F}%
_{t}}\int_{t}^{T}\tilde{Y}\langle \tilde{Z},K(Y)\rangle ds \\
&\leq &C_{1}^{2}+2C_{1}\mathbb{E}^{\mathcal{F}_{t}}\int_{t}^{T}|\langle 
\tilde{Z},K(Y)\rangle |ds
\end{eqnarray*}%
then integrating over $\mathbb{T}^{2}$ and using the estimate from the
maximum principle, we have

\begin{eqnarray*}
&&||\tilde{Y}_{t}||^{2}+2\nu \mathbb{E}^{\mathcal{F}_{t}}\int_{t}^{T}||%
\tilde{Z}||^{2}ds \\
&\leq &C_{1}^{2}+2C_{1}\mathbb{E}^{\mathcal{F}_{t}}\int_{t}^{T}\int_{\mathbb{%
T}^{2}}|\langle \tilde{Z},K(Y)\rangle |ds \\
&\leq &C_{1}^{2}+2C_{1}\mathbb{E}^{\mathcal{F}_{t}}\int_{t}^{T}||K(Y)||||%
\tilde{Z}||ds \\
&\leq &C_{1}^{2}+C_{1}\mathbb{E}^{\mathcal{F}_{t}}\int_{t}^{T}\left[
\varepsilon ||K(Y)||^{2}+\frac{1}{\varepsilon }||\tilde{Z}||^{2}\right] ds \\
&\leq &C_{1}^{2}+\varepsilon C_{1}C_{0}\mathbb{E}^{\mathcal{F}_{t}}\left[
\int_{t}^{T}||Y||^{2}ds\right] +\frac{C_{1}}{\varepsilon }\mathbb{E}^{%
\mathcal{F}_{t}}\left[ \int_{t}^{T}||\tilde{Z}||^{2}ds\right] \text{.}
\end{eqnarray*}%
for every $\varepsilon >0$. Recall that 
\begin{equation*}
||\tilde{Z}||_{BMO}^{2}=\text{ess}\sup_{\Omega \times \lbrack 0,T]}\mathbb{E}%
^{\mathcal{F}_{t}}\int_{t}^{T}||\tilde{Z}||^{2}ds\text{ .}
\end{equation*}%
It follows that 
\begin{equation}
||\tilde{Z}||_{BMO}^{2}\leq \frac{C_{1}}{2\nu }\left[ C_{1}+T\varepsilon
C_{1}^{2}C_{0}+\frac{1}{\varepsilon }||\tilde{Z}||_{BMO}^{2}\right] \text{.}
\label{6-4-1}
\end{equation}%
Choose $\varepsilon =\frac{C_{1}}{\nu }$, we obtain 
\begin{equation*}
||\tilde{Z}||_{BMO}\leq \frac{C_{1}}{\nu }\sqrt{\nu +TC_{0}C_{1}^{2}}\text{ .%
}
\end{equation*}

That is, the norms $||\tilde{Y}||_{\infty }$ and $||\tilde{Z}||_{BMO}$ are
uniformly bounded, depending only on $\nu ,C_{1}$, $C_{0}$ and $T$.

\subsection{Contraction property}

Let $\alpha $ be a real number to be chosen later, and consider $%
Y_{t}^{\alpha }=e^{\alpha t}Y_{t}$ and $\tilde{Y}_{t}^{\alpha }=e^{\alpha t}%
\tilde{Y}_{t}$. Then, according to integration by parts 
\begin{equation*}
d\tilde{Y}^{\alpha }=\langle \tilde{Z},K(Y^{\alpha })\rangle ds+\sqrt{2\nu }%
\langle \tilde{Z}^{\alpha },dB\rangle +\alpha \tilde{Y}^{\alpha }dt\text{.}
\end{equation*}

Denoting $\delta Y^{\alpha }=Y^{\alpha }-Y^{\prime \alpha }$ and $\delta
Z^{\alpha }=Z^{\alpha }-Z^{\prime \alpha }$. Then 
\begin{equation*}
d(\delta \tilde{Y}^{\alpha })=\Phi ds+\alpha (\delta \tilde{Y}^{\alpha })ds+%
\sqrt{2\nu }\langle \delta \tilde{Z}^{\alpha },dB\rangle
\end{equation*}%
where 
\begin{equation*}
\Phi _{s}=\langle \tilde{Z}_{s},K(Y_{s}^{\alpha })\rangle -\langle \tilde{Z}%
_{s}^{\prime },K(Y_{s}^{\prime \alpha })\rangle \text{ .}
\end{equation*}%
It follows that 
\begin{eqnarray*}
|\delta \tilde{Y}_{t}^{\alpha }|^{2} &=&-2\nu \int_{t}^{T}|\delta \tilde{Z}%
^{\alpha }|^{2}ds-2\alpha \int_{t}^{T}|\delta \tilde{Y}^{\alpha }|^{2}ds \\
&&-2\int_{t}^{T}(\delta \tilde{Y}^{\alpha })\Phi ds-2\sqrt{2\nu }%
\int_{t}^{T}(\delta \tilde{Y}^{\alpha })\langle \delta \tilde{Z}^{\alpha
},dB\rangle
\end{eqnarray*}%
by taking conditional expectation given the information up to $\mathcal{F}%
_{t}$ to obtain 
\begin{eqnarray*}
|\delta \tilde{Y}_{t}^{\alpha }|^{2} &=&-2\nu \mathbb{E}^{\mathcal{F}%
_{t}}\int_{t}^{T}|\delta \tilde{Z}^{\alpha }|^{2}ds-2\alpha \mathbb{E}^{%
\mathcal{F}_{t}}\int_{t}^{T}|\delta \tilde{Y}^{\alpha }|^{2}ds \\
&&-2\mathbb{E}^{\mathcal{F}_{t}}\int_{t}^{T}(\delta \tilde{Y}^{\alpha })\Phi
ds\text{.}
\end{eqnarray*}%
Now integrating over $\mathbb{T}^{2}$, to obtain 
\begin{eqnarray}
||\delta \tilde{Y}_{t}^{\alpha }||^{2} &=&-2\nu \mathbb{E}^{\mathcal{F}%
_{t}}\int_{t}^{T}||\delta \tilde{Z}^{\alpha }||^{2}ds-2\alpha \mathbb{E}^{%
\mathcal{F}_{t}}\int_{t}^{T}||\delta \tilde{Y}^{\alpha }||^{2}ds  \notag \\
&&-2\mathbb{E}^{\mathcal{F}_{t}}\int_{t}^{T}\int_{\mathbb{T}^{2}}(\delta 
\tilde{Y}^{\alpha })\Phi ds  \label{13-2-4}
\end{eqnarray}%
Let us write for simplicity 
\begin{equation*}
J(t)=||\delta \tilde{Y}_{t}^{\alpha }||^{2}+2\nu \mathbb{E}^{\mathcal{F}%
_{t}}\int_{t}^{T}||\delta \tilde{Z}^{\alpha }||^{2}ds+2\alpha \mathbb{E}^{%
\mathcal{F}_{t}}\int_{t}^{T}||\delta \tilde{Y}^{\alpha }||^{2}ds\text{.}
\end{equation*}%
Then (\ref{13-2-4}) implies that 
\begin{eqnarray}
J(t) &=&-2\mathbb{E}^{\mathcal{F}_{t}}\int_{t}^{T}\int_{\mathbb{T}%
^{2}}(\delta \tilde{Y}^{\alpha })\Phi ds  \notag \\
&\leq &2\mathbb{E}^{\mathcal{F}_{t}}\int_{t}^{T}||\delta \tilde{Y}^{\alpha
}||||\Phi ||ds  \notag \\
&\leq &2\left( \mathbb{E}^{\mathcal{F}_{t}}\int_{t}^{T}||\delta \tilde{Y}%
^{\alpha }||^{2}ds\right) ^{\frac{1}{2}}\left( \mathbb{E}^{\mathcal{F}%
_{t}}\int_{t}^{T}||\Phi ||^{2}ds\right) ^{\frac{1}{2}}  \label{13-02-11}
\end{eqnarray}%
which yields that 
\begin{eqnarray}
J(t) &\leq &2\left( \mathbb{E}^{\mathcal{F}_{t}}\int_{t}^{T}||\delta \tilde{Y%
}^{\alpha }||^{2}ds\right) ^{\frac{1}{2}}\left( \mathbb{E}^{\mathcal{F}%
_{t}}\int_{t}^{T}||\Phi ||^{2}ds\right) ^{\frac{1}{2}}  \notag \\
&\leq &2\alpha \mathbb{E}^{\mathcal{F}_{t}}\int_{t}^{T}||\delta \tilde{Y}%
^{\alpha }||^{2}ds+\frac{1}{2\alpha }\mathbb{E}^{\mathcal{F}%
_{t}}\int_{t}^{T}||\Phi ||^{2}ds\text{.}  \label{13-02-26-1}
\end{eqnarray}

Let us now consider the last integral appearing on the right-hand side of (%
\ref{13-02-26-1}). It is clear that 
\begin{eqnarray*}
||\Phi _{s}|| &=&||\tilde{Z}_{s}\cdot K(Y_{s}^{\alpha })-\tilde{Z}%
_{s}^{\prime }\cdot K(Y_{s}^{\prime \alpha })|| \\
&=&||\tilde{Z}_{s}\cdot K(\delta Y_{s}^{\alpha })+\delta \tilde{Z}%
_{s}^{\alpha }\cdot K(Y_{s}^{\prime })|| \\
&\leq &||K(\delta Y_{s}^{\alpha })||||\tilde{Z}_{s}||+||K(Y_{s}^{\prime
})||||\delta \tilde{Z}_{s}^{\alpha }|| \\
&\leq &C_{0}||\delta Y_{s}^{\alpha }||||\tilde{Z}_{s}||+C_{0}||Y_{s}^{\prime
}||||\delta \tilde{Z}_{s}^{\alpha }||
\end{eqnarray*}%
plugging into (\ref{13-02-26-1}) we conclude that 
\begin{eqnarray}
J(t) &\leq &2\alpha \mathbb{E}^{\mathcal{F}_{t}}\int_{t}^{T}||\delta \tilde{Y%
}^{\alpha }||^{2}ds+\frac{C_{0}^{2}C_{1}^{2}}{\alpha }\frac{\nu
+TC_{0}C_{1}^{2}}{\nu ^{2}}||\delta Y^{\alpha }||_{\infty }^{2}  \notag \\
&&+\frac{C_{0}^{2}C_{1}^{2}}{\alpha }\mathbb{E}^{\mathcal{F}%
_{t}}\int_{t}^{T}||\delta \tilde{Z}^{\alpha }||^{2}ds  \label{13-02-26-02}
\end{eqnarray}%
where we have used the uniform bounds

\begin{equation*}
||\tilde{Y}||_{\infty }\leq C_{1}\text{ and }||\tilde{Z}||_{BMO}\leq \frac{%
C_{1}}{\nu }\sqrt{\nu +TC_{0}C_{1}^{2}}\text{ .}
\end{equation*}
Choose $\alpha >0$ such that 
\begin{equation*}
\frac{C_{0}^{2}C_{1}^{2}}{\alpha }\frac{\nu +TC_{0}C_{1}^{2}}{\nu ^{2}}\leq 
\frac{1}{16}\text{, }\frac{C_{0}^{2}C_{1}^{2}}{\alpha }\leq \frac{\nu }{4}
\end{equation*}
then (\ref{13-02-26-02}) yields that 
\begin{equation*}
||\delta \tilde{Y}^{\alpha }||_{\infty }+||\delta \tilde{Z}^{\alpha
}||_{BMO}\leq \frac{1}{2}||\delta Y^{\alpha }||_{\infty }\text{ .}
\end{equation*}

\begin{theorem}
\label{th-3}There is $\alpha >0$ such that, $\mathcal{L}$ is a contraction
on $\mathcal{H}$ under the norm 
\begin{equation*}
||Y||_{\alpha ,\infty }=||Y^{\alpha }||_{\infty }+||Z^{\alpha }||_{BMO}
\end{equation*}%
where $Z_{t}^{\alpha }=e^{\alpha t}Z_{t}$ and $Z$ is the density process of
the martingale part of $Y$.
\end{theorem}

We are now in a position to complete the proof of Theorem \ref{main-th}. The
sequence of Picard's iteration is constructed as the following. Begin with 
\begin{equation*}
Y_{0}(t,x)=\mathbb{E}\left\{ \left. \xi (x)\right| \mathcal{F}_{t}\right\}
\end{equation*}
(here we mean the continuous version of the optional projection of $\xi $)
and $Z_{0}$ is the density process of $Y_{0}$ with respect to the Brownian
motion determined by It\^{o}'s martingale representation. Since $\xi \in
W^{2,2}(\mathbb{T}^{2})$, so $Y_{0}(t,\cdot )\in W^{2,2}(\mathbb{T}^{2})$
for all $t$ almost surely. Define $Y_{n+1}=\mathcal{L}(Y_{n})$ for $%
n=0,1,2,\cdots $. Then Lemma \ref{lem-q1} implies that all $Y_{n}\in 
\mathcal{H}$ and in particular $(t,x)\rightarrow Y_{n}(\cdot ,t,x)$ are
continuous almost surely, so that 
\begin{equation}
\mathbb{P}\{|Y_{n}(t,x)|\leq C_{1}\text{ for all }(t,x,n)\in \lbrack
0,T]\times \mathbb{T}^{2}\times \mathbb{N}\}=1\text{.\ }  \label{bd-01}
\end{equation}
Theorem \ref{th-3} implies that $\{Y_{n}\}$ is a Cauchy sequence under the
norm $||\cdot ||_{\alpha ,\infty }$ for some $\alpha >0$, and therefore has
a limit $Y$ which is a solution to (\ref{b-12-01}).

\providecommand{\bysame}{\leavevmode\hbox to3em{\hrulefill}\thinspace} %
\providecommand{\MR}{\relax\ifhmode\unskip\space\fi MR } 
\providecommand{\MRhref}[2]{  \href{http://www.ams.org/mathscinet-getitem?mr=#1}{#2}
} \providecommand{\href}[2]{#2}


\end{document}